\documentclass[12pt]{article}
\usepackage[utf8]{inputenc}

\usepackage{amsmath,amssymb,amsthm,mathtools,amsfonts}
\usepackage{stmaryrd}
\usepackage{mathrsfs}

\usepackage{graphicx}
\usepackage{xcolor}
\usepackage{multicol}
\usepackage{comment}
\usepackage[toc]{appendix}
\usepackage{lipsum}
\usepackage{indentfirst}

\usepackage[top=2.4cm,bottom=2.2cm,left=2.6cm,right=2cm]{geometry}

	\usepackage[shortlabels]{enumitem}
	
	\usepackage{url}
	\usepackage{hyperref}
	
	\allowdisplaybreaks
	\setlist[enumerate,1]{label={\normalfont(\roman*)}}
	
	\def\0{{\bf 0}}
	\def\1{{\bf 1}}

	\newcommand{\cH}{{\mathcal {H}}}

	\theoremstyle{plain}
	\newtheorem{thm}{Theorem}
	\newtheorem{lem}[thm]{Lemma}
	\newtheorem{rem}{Remark}
	\newtheorem{prp}[thm]{Proposition}

	\theoremstyle{definition}

	\newtheorem{definition}[thm]{Definition}

\begin{document}
		
		\title{Wang-Qiu-Hu switching and isomorphism}
		\author{
Aida Abiad\footnotemark[1]
\and
Hong-Jun Ge\footnotemark[2]
}

\date{}
\maketitle
\begingroup
\renewcommand{\thefootnote}{\fnsymbol{footnote}}
\footnotetext[1]{ Department of Mathematics and Computer Science, Eindhoven University of Technology, The Netherlands, Department of Mathematics and Data Science, Vrije Universiteit Brussel, Belgium.
Email: \href{mailto:a.abiad.monge@tue.nl}{\texttt{a.abiad.monge@tue.nl}}.}

\footnotetext[2]{School of Mathematical Sciences, University of Science and Technology of China, Hefei, China.
Email: \href{mailto:gehj22@mail.ustc.edu.cn}{\texttt{gehj22@mail.ustc.edu.cn}}.}
		\endgroup

\setcounter{footnote}{0}
		
\begin{abstract}
Cospectral graphs (graphs that share the same eigenvalues) expose the limitations of using the graph spectrum to uniquely identify graphs, and they also help to understand what structural properties a graph spectrum cannot capture. Switching methods, which are standard tools for constructing cospectral graphs,  require specific structural and algebraic conditions to hold for the operation to preserve the graph's spectrum. However, there is no guarantee that the obtained cospectral switched graph is non-isomorphic. In this paper we study this isomorphism problem for a recent and prolific switching method to produce cospectral graphs with respect to the adjacency spectra: Wang-Qiu-Hu  (WQH) switching. We do so by using common-neighbour multisets associated with a WQH partition, which allows us to derive an external common-neighbour criterion for certifying non-isomorphism after WQH-switching. Then, we apply the new criterion to clique extensions and to weak tensor products, with coclique extensions as a special case. As an application we obtain infinite families of cospectral non-isomorphic graphs, including some known constructions. Finally we extend the conditions of WQH-switching to generalized adjacency matrices and, under an additional degree condition, to Laplacian and signless Laplacian matrices.\\

\end{abstract}

		\section{Introduction}

	Spectral graph theory studies the extent to which structural properties of a
graph are encoded in the spectra of associated matrices.  Cospectral
non-isomorphic graphs mark the limitations of such spectral characterizations.
Switching methods, including Godsil--McKay switching \cite{GM82},
Wang--Qiu--Hu switching \cite{WQH2019,QIU2020265}, and
Abiad--Haemers switching \cite{AHswitching} (see also \cite{abiad2024,MAO2023}), are standard
tools for constructing cospectral graphs. A switching method requires a prescribed switching set or switching partition.
The existence of such a partition guarantees cospectrality after the prescribed
edge changes, but it does not by itself guarantee a new graph: the switched
graph may still be isomorphic to the original one.  A central challenge is
therefore to give verifiable conditions under which switching yields a
non-isomorphic cospectral mate.  For Godsil--McKay switching, this question was investigated in \cite{abh2015}.

WQH-switching has been used to show that certain graph classes are not
determined by their spectrum \cite{MRC,bch2015,h1996,h2020}, to construct
strongly regular graphs \cite{ABIAD20231,IM19,IHRINGER2021112560,QJW20}, and
more recently in the construction of Neumaier graphs
\cite{EVANS2023113384}.  In all of these applications, proving non-isomorphism after switching is
often the most challenging part of the proof argument.

This paper develops non-isomorphism criteria for WQH-switching.  We introduce
common-neighbour multisets associated with a WQH partition and prove an
external common-neighbour test: a change in this external data certifies that
the switched graph is non-isomorphic to the original one.  We apply this test
to two graph operations.  First, we prove that, under explicit hypotheses on the
base WQH partition, every \(t\)-clique extension with \(t\ge2\) admits a
WQH-switching that produces a cospectral non-isomorphic graph; see
Theorem~\ref{thm:clique-wqh}.  Second, we prove an analogous result for
WQH-switching on one fibre of a weak tensor product \(\mathcal H\times G\),
with coclique extensions included as a special case; see
Theorem~\ref{thm:WQH-weak-tensor}.  These results are then specialized in
Section~\ref{sec:applications} to weak tensor products, to \(t\)-coclique
extensions, and to clique-extension constructions whose twists are realized by
WQH-switchings.  We also record extensions to generalized adjacency matrices
and, under an additional degree condition, to Laplacian and signless Laplacian
matrices.

	The paper is organized as follows.  Section~\ref{sec:preliminaries} recalls
WQH-switching, introduces the common-neighbour notation, and proves the
external common-neighbour test used throughout the paper.  Section~\ref{sec:cliqueextension}
covers clique extensions, while Section~\ref{sec:weak-tensor} investigates weak tensor
products. Section~\ref{sec:other-matrices} extends the WQH-switching argument to generalized adjacency matrices and, under an additional degree condition, to
Laplacian and signless Laplacian matrices.   Finally, Section~\ref{sec:applications} presents applications of the new isomorphism criteria to weak tensor products, coclique extensions, and
clique-extension constructions.

\section{Preliminaries}\label{sec:preliminaries}

Let $[t]:=\{1,2,\ldots,t\}$.  All graphs are finite.  Unless explicitly stated
otherwise, graphs are simple.  The symbol $\sqcup$ denotes disjoint union, and
$\otimes$ denotes the Kronecker product of matrices.

For a graph $X$ and a vertex $x\in V(X)$, let $N_X(x)$ denote the neighbourhood
of $x$.  If $U\subseteq V(X)$, put
\[
d_U(x):=|N_X(x)\cap U|.
\]
In particular, $d_X(x):=d_{V(X)}(x)$.  For vertices $u,v\in V(X)$, write
\[
\lambda_X(u,v):=|N_X(u)\cap N_X(v)|.
\]
Thus, for simple graphs, $\lambda_X(u,u)=d_X(u)$.  All multisets of the form
$\{\lambda_X(u,v):u\in U,\ v\in W\}$ are counted over ordered pairs
$(u,v)\in U\times W$.  For $U,W\subseteq V(X)$ and an integer $r$, define
\[
\Lambda_X(U,W):=\{\lambda_X(u,v):u\in U,\ v\in W\},
\]
\[
\Lambda_{r,X}(U,W):=
\{\lambda_X(u,v):u\in U,\ d_X(u)=r,\ v\in W\}.
\]
We also use
\[
\overline\Lambda_X(U):=\Lambda_X(U,V(X)\setminus U).
\]

\begin{definition}
Let $G$ be a graph.  The \emph{$t$-clique extension} of $G$ is the graph with
vertex set $V(G)\times[t]$, where $(u,i)$ and $(v,j)$ are adjacent if and only if
either $u=v$ and $i\ne j$, or $uv\in E(G)$.
\end{definition}

\subsection{WQH-switching}

In 2019, Wang, Qiu, and Hu~\cite{WQH2019} introduced a new switching, which generalizes the well-known GM-switching proposed by Godsil and McKay~\cite{GM82}.

\begin{definition}\label{def:wqh-type}
Let $G$ be a graph with vertex set
\[
V(G)=C_1\sqcup C_2\sqcup D_1\sqcup D_2\sqcup D_3.
\]
We say that $G$ is of \emph{WQH-type} with respect to this partition if the
following conditions hold:
\begin{enumerate}[\textup{(P\arabic*)}]
\item $|C_1|=|C_2|$;
\item $d_{C_1}(u)-d_{C_2}(u)=d_{C_2}(v)-d_{C_1}(v)$ for every $u\in C_1$ and
$v\in C_2$;
\item for every $x\in C_1$,
$N_G(x)\cap(D_1\cup D_2)=D_1$;
\item for every $y\in C_2$,
$N_G(y)\cap(D_1\cup D_2)=D_2$;
\item for every $w\in D_3$, one has $d_{C_1}(w)=d_{C_2}(w)$.
\end{enumerate}

The \emph{WQH-switched graph} $G'$ with respect to $(G,C_1,C_2)$ is obtained
from $G$ by replacing \textup{(P3)} and \textup{(P4)} by
\[
N_{G'}(x)\cap(D_1\cup D_2)=D_2\qquad (x\in C_1),
\]
\[
N_{G'}(y)\cap(D_1\cup D_2)=D_1\qquad (y\in C_2),
\]
and leaving all other adjacencies unchanged.
\end{definition}

Let \(q:=|C_1|=|C_2|\).  When \(q\in\{1,2\}\), WQH-switching is equivalent
to GM-switching, see \cite{ih2022}. 
For some graph products, Abiad, Brouwer, and Haemers \cite{abh2015} provided sufficient conditions for being non-isomorphic after GM-switching.
So, in our work, we only focus on $|C_1|=|C_2|\geq 3$.
 For arbitrary \(q\), the corresponding WQH-switching matrix
with respect to the partition \((C_1,C_2,D_1\cup D_2\cup D_3)\) is
\begin{equation}\label{eq:wqh-matrix}
Q=
\begin{bmatrix}
	I_q-\frac1qJ_q & \frac1qJ_q & 0\\
	\frac1qJ_q & I_q-\frac1qJ_q & 0\\
	0&0&I
\end{bmatrix},
\end{equation}
where \(I_q\) is the \(q\times q\) identity matrix and \(J_q\) is the
\(q\times q\) all-one matrix.  Then $Q^\top=Q$ and $Q^2=I$.  If $A$ is the
adjacency matrix of $G$, then the adjacency matrix of $G'$ is $A':=QAQ.$

\begin{thm}[WQH-switching, {\cite[Theorem~3.5]{WQH2019}}]\label{thm:wqh-cospectral}
Let $G$ be a WQH-type graph with WQH-switched graph $G'$.  Then $G$ and $G'$
are cospectral.
\end{thm}

A more general version of WQH-switching was introduced in \cite{QJW20}.
Recently, several alternative switching methods to construct cospectral graphs have been proposed in the literature, see e.g. \cite{abiad2024,MAO2023}.

The next lemma records some common-neighbour conditions that must hold if a
WQH-type graph is isomorphic to its switched graph.  
We state these conditions
in a form that will be used in Section~\ref{sec:cliqueextension} and \ref{sec:weak-tensor} to detect non-isomorphism.

\begin{lem}\label{lem:external-test}
Let $G$ be a WQH-type graph with vertex partition
$C_1\sqcup C_2\sqcup D_1\sqcup D_2\sqcup D_3$.  Put
$S:=C_1\cup C_2$ and $Y:=V(G)\setminus S.$
Let $G'$ be the WQH-switched graph with respect to $(G,C_1,C_2)$.  Assume that
$|D_1|=|D_2|$.  Then the following statements hold.
\begin{enumerate}[\textup{(\alph*)}]
\item For every integer $r$,
$
\Lambda_{r,G}(S,S)=\Lambda_{r,G'}(S,S)
$ and $
\Lambda_{r,G}(Y,Y)=\Lambda_{r,G'}(Y,Y).$
\item If there exists an integer $r$ such that
\[
\Lambda_{r,G}(S,Y)\sqcup\Lambda_{r,G}(Y,S)
\ne
\Lambda_{r,G'}(S,Y)\sqcup\Lambda_{r,G'}(Y,S),
\]
then $G$ and $G'$ are non-isomorphic.
\item If
$\overline\Lambda_G(S)\ne\overline\Lambda_{G'}(S)$,
then $G$ and $G'$ are non-isomorphic.
\end{enumerate}
\end{lem}

\begin{proof}
Since $|D_1|=|D_2|$, WQH-switching preserves the degree of every vertex in
$S$.  It also preserves the degree of every vertex in $Y$, because vertices in
$D_1$ and $D_2$ interchange the equally large sets $C_1$ and $C_2$, while
vertices in $D_3$ are not switched.

We next compare common-neighbour numbers inside $S$ and inside $Y$.  For
$x,y\in S$, the contribution from $S\cup D_3$ is unchanged.  The contribution
from $D_1\cup D_2$ is $|D_1|$ for two vertices in $C_1$, $|D_2|$ for two
vertices in $C_2$, and $0$ for one vertex in each of $C_1$ and $C_2$; after
switching these first two values are interchanged.  Since $|D_1|=|D_2|$, we have
\[
\lambda_G(x,y)=\lambda_{G'}(x,y)\qquad (x,y\in S).
\]
For $x,y\in Y$, the contribution from $Y$ is unchanged.  The contribution from
$S$ is also unchanged: for one vertex in $D_1$ and one in $D_2$ it is $0$ before
and after switching; for two vertices in $D_1$ it changes from $|C_1|$ to
$|C_2|$, and the case of two vertices in $D_2$ is symmetric; if one vertex lies
in $D_3$, equality follows from $d_{C_1}(z)=d_{C_2}(z)$ for $z\in D_3$; and if
both vertices lie in $D_3$, the $S$-neighbourhoods are unchanged.  Hence
\[
\lambda_G(x,y)=\lambda_{G'}(x,y)\qquad (x,y\in Y).
\]
Together with preservation of degrees, this proves (a).

For (b), consider the graph invariant
\[
\Lambda_{r,X}:=\{\lambda_X(u,v):u,v\in V(X),\ d_X(u)=r\}.
\]
It decomposes as the disjoint union of the four ordered parts
$\Lambda_{r,X}(S,S)$, $\Lambda_{r,X}(S,Y)$, $\Lambda_{r,X}(Y,S)$, and
$\Lambda_{r,X}(Y,Y)$.  By (a), the two internal parts are unchanged under
switching.  Therefore a change in the two external parts forces
$\Lambda_{r,G}\ne\Lambda_{r,G'}$, and hence $G\not\cong G'$.

For (c), use the invariant
$\Lambda_X:=\{\lambda_X(u,v):u,v\in V(X)\}$.  The internal parts on $S\times S$ and
$Y\times Y$ are unchanged by the preceding argument.  Since
$\lambda_X(u,v)=\lambda_X(v,u)$, a change in $\overline\Lambda_X(S)=\Lambda_X(S,Y)$
forces a change in the full ordered multiset $\Lambda_X$.  Hence $G$ and $G'$ are
non-isomorphic.
\end{proof}

\begin{rem}
The equalities forced by Lemma~\ref{lem:external-test} are necessary conditions
for an isomorphism $G\cong G'$, but they should not be expected to be sufficient.
They only record common-neighbour data relative to the prescribed switching set
$S=C_1\cup C_2$, and such data is generally too coarse to force an isomorphism.
Moreover, even when $G$ and $G'$ are isomorphic, an isomorphism need not preserve
$S$ setwise.  Indeed, Abiad, Van de Berg and Simoens constructed examples in
which WQH-switching produces an isomorphic graph but no isomorphism fixes the
switching set setwise; see \cite[Appendix]{ABScounting}.  Thus, in what follows,
we use Lemma~\ref{lem:external-test} only in the direction of detecting changes in
these invariants, which certifies non-isomorphism.
\end{rem}

Note also that the conditions from the previous lemma  also hold in the more general framework of design switching introduced in \cite{ihringer2025}. In particular, one can state a general invariant-based condition, however this is a broad sufficient condition rather than a sharp design-switching-specific theorem.

\begin{rem}
The same philosophy applies to any switching obtained by conjugating the
adjacency matrix with a regular orthogonal matrix supported on a switching set.
Let $X$ be a graph with vertex set $C\sqcup Y$, and let $X'$ be obtained from
$X$ by such a switching on $C$.  If, for some integer $r$, the internal
degree-refined common-neighbour multisets on $C\times C$ and $Y\times Y$ are
unchanged, but the external part changes, that is,
\[
\Lambda_{r,X}(C,C)=\Lambda_{r,X'}(C,C),\qquad
\Lambda_{r,X}(Y,Y)=\Lambda_{r,X'}(Y,Y),
\]
while
\[
\Lambda_{r,X}(C,Y)\sqcup\Lambda_{r,X}(Y,C)
\ne
\Lambda_{r,X'}(C,Y)\sqcup\Lambda_{r,X'}(Y,C),
\]
then $X$ and $X'$ are non-isomorphic.  Indeed, the full multiset
\[
\{\lambda_X(u,v):u,v\in V(X),\ d_X(u)=r\}
\]
is an isomorphism invariant.  Lemma~\ref{lem:external-test} is the WQH
specialization of this obstruction, where the required internal equalities
follow from $|D_1|=|D_2|$ and the WQH balance conditions.
\end{rem}

\section{Clique extensions of WQH-type graphs}\label{sec:cliqueextension}

Throughout this section, let $G$ be a WQH-type graph with vertex partition
$C_1\sqcup C_2\sqcup D_1\sqcup D_2\sqcup D_3$.  Put
$S:=C_1\cup C_2.$
For $x\in V(G)$ and $a\in[t]$, write $x^{(a)}:=(x,a)$.  For $x\in S$, define its
$D_3$-type by
\[
\tau(x):=N_G(x)\cap D_3.
\]
For $i\in\{1,2\}$, an integer $m$, and $A\subseteq D_3$, define
\[
f_{i,m}(A):=\#\{x\in C_i: d_G(x)=m,\ \tau(x)=A\}.
\]
\begin{thm}\label{thm:clique-wqh}
	Let $G$ be a WQH-type graph whose vertex set is partitioned as
	$C_1\sqcup C_2\sqcup D_1\sqcup D_2\sqcup D_3$.  Let $t\ge 2$, and let
	$H$ be the $t$-clique extension of $G$.  Put
	\[
	C_1^H:=C_1\times\{1\},\qquad C_2^H:=C_2\times\{1\},
	\]
	\[
	D_1^H:=D_1\times[t],\qquad D_2^H:=D_2\times[t],
	\]
	\[
	D_3^H:=\bigl(S\times\{2,\ldots,t\}\bigr)\sqcup(D_3\times[t]).
	\]
	Assume that:
	\begin{enumerate}[\textup{(C\arabic*)}]
		\item $|D_1|=|D_2|>0$, $|C_1|=|C_2|\ge 3$, and $S=C_1\cup C_2$ is a clique;
		\item there exists an integer $s$ such that
		$\bigl(f_{1,s}(A)\bigr)_{A\subseteq D_3}
		\ne
		\bigl(f_{2,s}(A)\bigr)_{A\subseteq D_3};$
		\item no vertex of $D_3$ is adjacent to every vertex of $S$.
	\end{enumerate}
	Then $H$ is of WQH-type with respect to the partition
	$C_1^H\sqcup C_2^H\sqcup D_1^H\sqcup D_2^H\sqcup D_3^H.$
	Let $H'$ be the WQH-switched graph of $H$.  Then $H$ and $H'$ are
	non-isomorphic.  In particular, they are cospectral non-isomorphic graphs.
\end{thm}

Throughout the
following lemmas, we work under the hypotheses of
Theorem~\ref{thm:clique-wqh}, and $H'$ denotes the graph obtained from $H$ by
interchanging the adjacencies between $C_1^H,C_2^H$ and $D_1^H,D_2^H$ according
to the displayed partition.

Let $\ell$ be an integer and set
\[
r:=t\ell+(t-1),
\qquad
M:=r-1=t\ell+t-2.
\]
If $d_G(x)=\ell$, then
\[
d_H(x^{(a)})=d_{H'}(x^{(a)})=r
\qquad (a\in[t]).
\]

\begin{lem}\label{lem:clique-D}
	Let $x\in C_i$, where $i\in\{1,2\}$, and let $y\in D_1\cup D_2\cup D_3$.  If
	$d_G(x)=\ell$, then, for $X\in\{H,H'\}$ and all $a,b\in[t]$,
	$\lambda_X(x^{(a)},y^{(b)})<M.$
\end{lem}

\begin{proof}
	It suffices to treat $x\in C_1$; the other case is symmetric.  We have
	$d_X(x^{(a)})=M+1$.
	
	First let $y\in D_1\cup D_2$.  In both $H$ and $H'$, the vertex $y^{(b)}$ is
	adjacent to exactly $t|C_1|$ vertices of $S\times[t]$.  Since $x^{(a)}$ is
	adjacent to all vertices of $S\times[t]$ except itself, at least
	$t|C_1|-1$ neighbours of $x^{(a)}$ are not adjacent to $y^{(b)}$.  Therefore
	\[
	\lambda_X(x^{(a)},y^{(b)})
	\le M+1-(t|C_1|-1)<M,
	\]
	because $t\ge 2$ and $|C_1|\ge 3$ imply $t|C_1|-1>1$.
	
	Now let $y\in D_3$.  By assumption \textup{(C3)}, $y$ is not adjacent to every
	vertex of $S$.  Since $d_{C_1}(y)=d_{C_2}(y)$ and $|C_1|=|C_2|$, there exists
	$z\in S\setminus\{x\}$ such that $zy\notin E(G)$.  As $S$ is a clique, all
	vertices $z^{(1)},\ldots,z^{(t)}$ are neighbours of $x^{(a)}$, while none of
	these vertices is adjacent to $y^{(b)}$.  Hence
	\[
	\lambda_X(x^{(a)},y^{(b)})\le M+1-t<M,
	\]
	because $t\ge2$.
\end{proof}

\begin{lem}\label{lem:clique-H}
	Let $x\in C_i$, where $i\in\{1,2\}$, and suppose that $d_G(x)=\ell$.  For
	$y\in V(G)$ and $a,b\in[t]$ with $x^{(a)}\ne y^{(b)}$, one has
	$\lambda_H(x^{(a)},y^{(b)})=M$
	if and only if
	$y\in C_i$ and $\tau(x)\subseteq\tau(y).$
\end{lem}

\begin{proof}
	Assume $x\in C_1$; the case $x\in C_2$ is symmetric.  By
	Lemma~\ref{lem:clique-D}, the equality cannot hold when
	$y\in D_1\cup D_2\cup D_3$.
	
	If $y\in C_2$, then every vertex of $D_1\times[t]$ is adjacent to $x^{(a)}$ and
	not adjacent to $y^{(b)}$.  Since $D_1\ne\varnothing$ and $t\ge 2$,
	\[
	\lambda_H(x^{(a)},y^{(b)})\le M+1-t<M.
	\]
	Thus equality can hold only when $y\in C_1$.
	
	Let $y\in C_1$.  If $y=x$, then $a\ne b$, and a direct count gives
	\[
	\lambda_H(x^{(a)},x^{(b)})=t\,d_G(x)+(t-2)=M.
	\]
	Now suppose that $y\ne x$.  Since $S$ is a clique, $x$ and $y$ are adjacent in
	$G$, and
	\[
	\lambda_H(x^{(a)},y^{(b)})=t\lambda_G(x,y)+2(t-1).
	\]
	As $d_G(x)=\ell$, this is equal to $M$ if and only if
	$\lambda_G(x,y)=\ell-1$.  Since $xy\in E(G)$, the latter equality is equivalent
	to
	$N_G(x)\setminus\{y\}\subseteq N_G(y).$
	For vertices $x,y\in C_1$, the neighbourhoods in $S$ and in $D_1$ already have
	this containment property, and the remaining condition is precisely
	$\tau(x)\subseteq\tau(y)$.
\end{proof}

\begin{lem}\label{lem:clique-Hprime}
	Let $x\in C_i$, where $i\in\{1,2\}$, and suppose that $d_G(x)=\ell$.  Let
	$y\in S$ and let $a,b\in[t]$ be such that exactly one of $a,b$ is equal to $1$.
	Then
	$\lambda_{H'}(x^{(a)},y^{(b)})=M$
	if and only if
	$y\in C_{3-i}$ and $\tau(x)\subseteq\tau(y).$
\end{lem}

\begin{proof}
	Assume $x\in C_1$; the case $x\in C_2$ is symmetric.  If $y\in C_1$, then the
	following obstruction occurs.  When $a=1$, all vertices of $D_2\times[t]$ are
	neighbours of $x^{(a)}$ and non-neighbours of $y^{(b)}$; when $b=1$, all vertices
	of $D_1\times[t]$ are neighbours of $x^{(a)}$ and non-neighbours of $y^{(b)}$.
	In either case,
	\[
	\lambda_{H'}(x^{(a)},y^{(b)})\le M+1-t<M,
	\]
	because $t\ge2$.  Thus equality can hold only when $y\in C_2$.
	
	Let $y\in C_2$.  The vertices $x^{(a)}$ and $y^{(b)}$ are adjacent, and
	$d_{H'}(x^{(a)})=M+1$.  Hence
	$\lambda_{H'}(x^{(a)},y^{(b)})=M$ if and only if every neighbour of $x^{(a)}$
	except $y^{(b)}$ is also a neighbour of $y^{(b)}$.  If $a=1$, both vertices are
	adjacent to all of $D_2\times[t]$; if $b=1$, both vertices are adjacent to all of
	$D_1\times[t]$.  In both cases their adjacencies to $S\times[t]$, except for
	the two vertices themselves, are forced by the clique $S$, and their adjacencies
	to $D_3\times[t]$ are governed by $\tau(x)$ and $\tau(y)$.  Therefore the above
	containment is equivalent to $\tau(x)\subseteq\tau(y)$.
\end{proof}

\begin{lem}\label{lem:clique-template}
	Fix $z\in D_1\cup D_2\cup D_3$ and $b\in[t]$.  Let $B\subseteq D_3$.  For
	$a\in\{1,2\}$, the value
	$\lambda_H(z^{(b)},y^{(1)})$
	is independent of the choice of $y\in C_a$ with $\tau(y)=B$; denote this value,
	when such a vertex exists, by $\alpha_a(B)$.  Moreover, for every
	$y\in C_a$ with $\tau(y)=B$,
	$\lambda_{H'}(z^{(b)},y^{(1)})=\alpha_{3-a}(B).$
\end{lem}

\begin{proof}
	For a vertex $y\in S$ with $\tau(y)=B$, the first-layer neighbourhood is
	\[
	N_H(y^{(1)})=
	\begin{cases}
		(S\times[t]\setminus\{y^{(1)}\})\cup(D_1\times[t])\cup(B\times[t]),
		& y\in C_1,\\[1mm]
		(S\times[t]\setminus\{y^{(1)}\})\cup(D_2\times[t])\cup(B\times[t]),
		& y\in C_2,
	\end{cases}
	\]
	and
	\[
	N_{H'}(y^{(1)})=
	\begin{cases}
		(S\times[t]\setminus\{y^{(1)}\})\cup(D_2\times[t])\cup(B\times[t]),
		& y\in C_1,\\[1mm]
		(S\times[t]\setminus\{y^{(1)}\})\cup(D_1\times[t])\cup(B\times[t]),
		& y\in C_2.
	\end{cases}
	\]
	These descriptions show first that $\lambda_H(z^{(b)},y^{(1)})$ depends only on
	the side of $y$ and on the set $B$.  Indeed, if $z\in D_1\cup D_2$, then the
	adjacency of $z$ to $S$ depends only on whether $y\in C_1$ or $y\in C_2$; if
	$z\in D_3$, then whether $z$ is adjacent to $y$ is determined by the condition
	$z\in B$.
	
	The same displayed neighbourhoods also give the exchange relation after
	switching.  For $z\in D_1$ or $z\in D_2$, the switch interchanges the first-layer
	sets $C_1\times\{1\}$ and $C_2\times\{1\}$ in the neighbourhood of $z^{(b)}$ and
	interchanges $D_1\times[t]$ and $D_2\times[t]$ in the neighbourhood of
	$y^{(1)}$.  This is exactly the effect of replacing the side $C_a$ by
	$C_{3-a}$ in the unswitched graph.  If $z\in D_3$, the neighbourhood of
	$z^{(b)}$ is unchanged, and the only side-dependent part of $N_{H'}(y^{(1)})$
	is again obtained by interchanging $D_1\times[t]$ and $D_2\times[t]$.  The
	possible exclusion of $y^{(1)}$ from $S\times[t]$ contributes in the same way
	because $z\in\tau(y)$ is determined by $B$.  Hence
	$\lambda_{H'}(z^{(b)},y^{(1)})=\alpha_{3-a}(B)$.
\end{proof}

\begin{proof}[Proof of Theorem~\ref{thm:clique-wqh}]
We first verify that $H$ is WQH-type with respect to the displayed partition.
Under assumption
\textup{(C1)}, the partition displayed in Theorem~\ref{thm:clique-wqh} is a
WQH-type partition of $H$: conditions (P1), (P3), and (P4) are immediate; (P2)
follows from the fact that $S$ is a clique and $|C_1|=|C_2|$; and (P5) follows
from the original WQH condition on $D_3$, while vertices in
$S\times\{2,\ldots,t\}$ see $|C_1|$ vertices in $C_1^H$ and $|C_2|$ vertices in
$C_2^H$.
 Hence $H$ is WQH-type with respect to the
displayed partition, and $H'$ is its WQH-switched graph.  By
Theorem~\ref{thm:wqh-cospectral}, $H$ and $H'$ are cospectral.  It remains to
prove that they are non-isomorphic.

	By assumption \textup{(C2)}, let $\ell$ be the largest integer such that
	\[
	\bigl(f_{1,\ell}(A)\bigr)_{A\subseteq D_3}
	\ne
	\bigl(f_{2,\ell}(A)\bigr)_{A\subseteq D_3}.
	\]
	Set
	\[
	r:=t\ell+(t-1),
	\qquad
	M:=r-1=t\ell+t-2,
	\]
	and write
	\[
	S_1:=S\times\{1\},
	\qquad
	D_H:=V(H)\setminus S_1.
	\]
	
	By Lemma~\ref{lem:external-test}~(b), applied to the WQH partition of $H$, it is
	enough to show that
	\[
	\Lambda_{r,H}(S_1,D_H)\sqcup\Lambda_{r,H}(D_H,S_1)
	\ne
	\Lambda_{r,H'}(S_1,D_H)\sqcup\Lambda_{r,H'}(D_H,S_1).
	\]
	We prove that the value $M$ occurs with larger multiplicity on the left-hand
	side than on the right-hand side.
	
	Let
	\[
	\kappa:=|S|-1+|D_1|=|S|-1+|D_2|.
	\]
	Since $S$ is a clique, for every $u\in S$ we have
	\begin{equation}\label{eq:degree-type}
		d_G(u)=\kappa+|\tau(u)|.
	\end{equation}
	For $A\subseteq D_3$, define
	\[
	F_i(A):=\#\{z\in C_i:A\subseteq\tau(z)\}\qquad (i=1,2).
	\]
	We claim that, whenever $f_{1,\ell}(A)+f_{2,\ell}(A)>0$,
	\begin{equation}\label{eq:F-difference}
		F_1(A)-F_2(A)=f_{1,\ell}(A)-f_{2,\ell}(A).
	\end{equation}
	Indeed, then every degree-$\ell$ vertex of type $A$ satisfies
	$|A|=\ell-\kappa$ by \eqref{eq:degree-type}.  If $B\supsetneq A$, then
	$|B|>\ell-\kappa$, so every vertex of type $B$ has degree larger than $\ell$.
	By the maximality of $\ell$, the numbers of vertices of type $B$ in $C_1$ and
	in $C_2$ are equal.  Hence all proper-superset contributions cancel in
	$F_1(A)-F_2(A)$, leaving only the degree-$\ell$ contribution of type $A$.  This
	proves \eqref{eq:F-difference}.
	
	Let $m_X$ be the multiplicity of $M$ in $\Lambda_{r,X}(S_1,D_H)$, where
	$X\in\{H,H'\}$.  By Lemmas~\ref{lem:clique-D} and~\ref{lem:clique-H}, if
	$x\in C_i$, $d_G(x)=\ell$, and $\tau(x)=A$, then for each $j\in\{2,\ldots,t\}$
	the number of vertices $y^{(j)}$ with
	$\lambda_H(x^{(1)},y^{(j)})=M$ is $F_i(A)$.  Therefore
	\begin{equation}\label{eq:mH}
		m_H=(t-1)\sum_{A\subseteq D_3}
		\bigl(f_{1,\ell}(A)F_1(A)+f_{2,\ell}(A)F_2(A)\bigr).
	\end{equation}
	Similarly, by Lemmas~\ref{lem:clique-D} and~\ref{lem:clique-Hprime},
	\begin{equation}\label{eq:mHprime}
		m_{H'}=(t-1)\sum_{A\subseteq D_3}
		\bigl(f_{1,\ell}(A)F_2(A)+f_{2,\ell}(A)F_1(A)\bigr).
	\end{equation}
	Subtracting \eqref{eq:mHprime} from \eqref{eq:mH} gives
	\[
	m_H-m_{H'}
	=(t-1)\sum_{A\subseteq D_3}
	\bigl(f_{1,\ell}(A)-f_{2,\ell}(A)\bigr)
	\bigl(F_1(A)-F_2(A)\bigr).
	\]
	For those $A$ with $f_{1,\ell}(A)=f_{2,\ell}(A)=0$, the summand is zero.  For
	the remaining $A$, use \eqref{eq:F-difference}.  Thus
	\begin{equation}\label{eq:m-positive}
		m_H-m_{H'}
		=(t-1)\sum_{A\subseteq D_3}
		\bigl(f_{1,\ell}(A)-f_{2,\ell}(A)\bigr)^2>0.
	\end{equation}
	Here the strict positivity uses $t\ge2$ and the choice of $\ell$.
	
	It remains to count the reverse ordered pairs.  Let $\mu_X$ be the multiplicity
	of $M$ in $\Lambda_{r,X}(D_H,S_1)$, where $X\in\{H,H'\}$.  Split the
	degree-$r$ vertices in $D_H$ into
	\[
	U_C:=\{x^{(j)}:x\in S,\ d_G(x)=\ell,\ j\in\{2,\ldots,t\}\},
	\]
	\[
	U_D:=\{z^{(j)}:z\in D_1\cup D_2\cup D_3,\ d_G(z)=\ell,\ j\in[t]\}.
	\]
	Write
	\[
	\mu_X=\nu_X^C+\nu_X^D,
	\]
	where $\nu_X^C$ and $\nu_X^D$ denote the contributions from $U_C$ and $U_D$,
	respectively.
	
	For $U_C$, the same argument as above, now using ordered pairs
	$(x^{(j)},y^{(1)})$ with $j\ge2$, gives
	\begin{equation}\label{eq:nuC-positive}
		\nu_H^C-\nu_{H'}^C
		=(t-1)\sum_{A\subseteq D_3}
		\bigl(f_{1,\ell}(A)-f_{2,\ell}(A)\bigr)^2.
	\end{equation}
	
	We now show that
	\begin{equation}\label{eq:nuD-equal}
		\nu_H^D=\nu_{H'}^D.
	\end{equation}
	Fix $u=z^{(j)}\in U_D$.  If $y\in S$ and $d_G(y)<\ell$, then
	\[
	d_H(y^{(1)})=d_{H'}(y^{(1)})=td_G(y)+(t-1)<M,
	\]
	so neither $\lambda_H(u,y^{(1)})$ nor $\lambda_{H'}(u,y^{(1)})$ can be equal
	to $M$.  If $d_G(y)=\ell$, then Lemma~\ref{lem:clique-D}, applied with the
	roles of the two vertices interchanged, again gives
	\[
	\lambda_H(u,y^{(1)})<M,
	\qquad
	\lambda_{H'}(u,y^{(1)})<M.
	\]
	
	It remains to consider vertices $y\in S$ with $d_G(y)>\ell$.  Let
	$B:=\tau(y)$.  By \eqref{eq:degree-type}, such a type $B$ satisfies
	$|B|>\ell-\kappa$.  Hence, by the maximality of $\ell$, the number of vertices
	of type $B$ in $C_1$ equals the number of vertices of type $B$ in $C_2$.  By
	Lemma~\ref{lem:clique-template}, for this fixed $u$ the switch interchanges the
	common-neighbour values attached to the two sides for each fixed type $B$.
	Since the two sides contain equally many vertices of type $B$, the total number
	of occurrences of the value $M$ contributed by degree larger than $\ell$ is the
	same in $H$ and $H'$.  This proves \eqref{eq:nuD-equal}.
	
	Combining \eqref{eq:nuC-positive} and \eqref{eq:nuD-equal}, we obtain
	\[
	\mu_H-\mu_{H'}
	=(t-1)\sum_{A\subseteq D_3}
	\bigl(f_{1,\ell}(A)-f_{2,\ell}(A)\bigr)^2>0.
	\]
	Together with \eqref{eq:m-positive}, this shows that the multiplicity of $M$ in
	\[
	\Lambda_{r,H}(S_1,D_H)\sqcup\Lambda_{r,H}(D_H,S_1)
	\]
	is strictly larger than its multiplicity in
	\[
	\Lambda_{r,H'}(S_1,D_H)\sqcup\Lambda_{r,H'}(D_H,S_1).
	\]
	By Lemma~\ref{lem:external-test}~(b), $H$ and $H'$ are non-isomorphic.
\end{proof}

\section{Weak tensor products}\label{sec:weak-tensor}

Let $G$ be a finite simple graph with adjacency matrix $A$.  Let $\cH$ be a
finite graph, possibly with loops but without multiple edges, with adjacency
matrix $A_{\cH}$.  The \emph{weak tensor product} $\cH\times G$ is the graph on
$V(\cH)\times V(G)$ with adjacency matrix
$$A_{\cH}\otimes A.$$
Equivalently, $(u,x)$ and $(v,y)$ are adjacent if and only if
$(A_{\cH})_{uv}=1$ and $xy\in E(G)$.  Since $A$ has zero diagonal,
$\cH\times G$ is simple even when $\cH$ has loops.  In particular, if
$\mathcal K_t^\circ$ is the complete graph on $t$ vertices with a loop at every
vertex, then $A_{\mathcal K_t^\circ}=J_t$, and hence
$A_{\mathcal K_t^\circ\times G}=J_t\otimes A$.  After the natural identification
of $V(\mathcal K_t^\circ)\times V(G)$ with $V(G)\times[t]$, this is precisely
the $t$-coclique extension of $G$.

For $u,v\in V(\cH)$, define
\[
\lambda_{\cH}^+(u,v):=(A_{\cH}^2)_{uv}.
\]
If $\cH$ is simple, then this is the usual common-neighbour number in $\cH$; in
general it counts length-two walks from $u$ to $v$.

In this section, let $G$ be WQH-type with vertex partition
$C_1\sqcup C_2\sqcup D_1\sqcup D_2\sqcup D_3$.  Put
$S:=C_1\cup C_2$ and $Y:=V(G)\setminus S$.
Let $G'$ be the WQH-switched graph with respect to $(G,C_1,C_2)$, and let $Q$
be the WQH-switching matrix.  Then the adjacency matrix of $G'$ is
$A':=QAQ.$
Set
\[
A^{\mathrm r}:=QA.
\]
For $i\in V(\cH)$, let $E_{ii}$ be the diagonal matrix unit corresponding to
$i$, and define
\begin{equation}\label{eq:Qhat}
\widehat Q_i:=I\otimes I+E_{ii}\otimes(Q-I).
\end{equation}
Thus $\widehat Q_i$ acts as $Q$ on the fibre $\{i\}\times V(G)$ and as the
identity on all other fibres.  We define $(\cH\times G)'_i$ to be the graph, if
it exists, with adjacency matrix
\[
\widehat Q_i(A_{\cH}\otimes A)\widehat Q_i.
\]
The theorem below proves that this matrix is indeed an adjacency matrix under
the stated assumptions.

\begin{thm}\label{thm:WQH-weak-tensor}
Let $G$ be a WQH-type graph with vertex partition
$C_1\sqcup C_2\sqcup D_1\sqcup D_2\sqcup D_3$, and keep the notation above.
Assume that:
\begin{enumerate}
\item[\rm(W1)] for every $u\in S$,
$d_{C_1}(u)=d_{C_2}(u)$;
\item[\rm(W2)]
$\overline\Lambda_G(S)=\overline\Lambda_{G'}(S)$;
\item[\rm(W3)] there exists $x_0\in S$ such that
$d_G(x_0)>\max_{x,y\in S}(A^{\mathrm r} A)_{xy}$.
\end{enumerate}
Let $\cH$ be a finite graph, possibly with loops but without multiple edges.  For
$i\in V(\cH)$, set
$X_i:=\{i\}\times S.$
If there exists $j\in V(\cH)$ such that $j\ne i$ and
$\lambda_{\cH}^+(i,j)>0$, then the following statements hold:
\begin{enumerate}[\textup{(\alph*)}]
\item The matrix defining $(\cH\times G)'_i$ is the adjacency matrix of the graph
obtained from $\cH\times G$ by WQH-switching on the pair
$(\{i\}\times C_1,\{i\}\times C_2)$.
\item The graphs $\cH\times G$ and $(\cH\times G)'_i$ are cospectral.
\item The external multiset changes:
$\overline\Lambda_{\cH\times G}(X_i)
\ne
\overline\Lambda_{(\cH\times G)'_i}(X_i).$
Moreover, if $|D_1|=|D_2|$, then $\cH\times G$ and $(\cH\times G)'_i$ are
non-isomorphic.
\end{enumerate}
\end{thm}

\begin{proof}
Let
$B:=A_{\cH}\otimes A$ and $B_i':=\widehat Q_iB\widehat Q_i.$
We first prove (a) and (b).

Let
\[
N_{\cH}(i):=\{h\in V(\cH):(A_{\cH})_{ih}=1\},
\]
where $i\in N_{\cH}(i)$ is allowed when $i$ has a loop.  Consider the following
partition of $V(\cH\times G)$:
\[
C_1':=\{i\}\times C_1,
\qquad
C_2':=\{i\}\times C_2,
\]
\[
D_1':=N_{\cH}(i)\times D_1,
\qquad
D_2':=N_{\cH}(i)\times D_2,
\]
\[
D_3':=V(\cH\times G)\setminus(C_1'\sqcup C_2'\sqcup D_1'\sqcup D_2').
\]
This is a WQH-type partition.  Conditions (P1), (P3), and (P4) follow directly
from the definition of the weak tensor product and from the WQH partition of
$G$.  Condition (P2) follows from (W1), since the only possible adjacencies
inside $C_1'\cup C_2'$ come from a loop at $i$.  For (P5), let $(h,z)\in D_3'$.
If $h\notin N_{\cH}(i)$, then $(h,z)$ has no neighbours in either $C_1'$ or
$C_2'$.  If $h\in N_{\cH}(i)$ and $z\in D_3$, then equality of the two numbers
of neighbours follows from the WQH condition $d_{C_1}(z)=d_{C_2}(z)$.  If
$h\in N_{\cH}(i)$ and $z\in S$, then it follows from (W1).

By (W1) and the explicit formula for $Q$, the row-switched matrix
$A^{\mathrm r}=QA$ is a $0$-$1$ matrix: on the rows indexed by $C_1\cup C_2$, it swaps the
incidences with $D_1$ and $D_2$ and leaves all other entries unchanged; outside
$C_1\cup C_2$ it is equal to $A$.  Since $Q$ is symmetric, $AQ=(A^{\mathrm r})^\top$ is
also a $0$-$1$ matrix.  The block form of $B_i'$, with blocks indexed by
$V(\cH)$, is
\begin{equation}\label{eq:Bi-prime}
(B_i')_{uv}=
\begin{cases}
(A_{\cH})_{ii}A', & u=v=i,\\[1mm]
(A_{\cH})_{iv}A^{\mathrm r}, & u=i,\ v\ne i,\\[1mm]
(A_{\cH})_{ui}(A^{\mathrm r})^\top, & u\ne i,\ v=i,\\[1mm]
(A_{\cH})_{uv}A, & u\ne i,\ v\ne i.
\end{cases}
\end{equation}
Hence $B_i'$ is a symmetric $0$-$1$ matrix.  Its diagonal is zero because
$A$ and $A'$ have zero diagonal.  Thus $B_i'$ is the adjacency matrix of a
simple graph.  From the WQH partition above and the same block description, this
graph is exactly the graph obtained from $\cH\times G$ by WQH-switching on
$(C_1',C_2')=(\{i\}\times C_1,\{i\}\times C_2)$.  This proves (a).

Since $\widehat Q_i^\top=\widehat Q_i$ and $\widehat Q_i^2=I$, the matrices
$B$ and $B_i'$ are similar.  Hence $\cH\times G$ and $(\cH\times G)'_i$ are
cospectral.  This proves (b).

It remains to prove (c).  Decompose
\[
\overline\Lambda_{\cH\times G}(X_i)=\mathcal M_1\sqcup\mathcal M_2\sqcup\mathcal M_3,
\]
where
\[
\mathcal M_1:=\{\lambda_{\cH\times G}((i,x),(i,y)):x\in S,\ y\in Y\},
\]
\[
\mathcal M_2:=\{\lambda_{\cH\times G}((i,x),(h,y)):x\in S,\ h\ne i,\ y\in Y\},
\]
\[
\mathcal M_3:=\{\lambda_{\cH\times G}((i,x),(h,y)):x\in S,\ h\ne i,\ y\in S\}.
\]
Define $\mathcal M_1',\mathcal M_2',\mathcal M_3'$ analogously for
$(\cH\times G)'_i$.

Since
\[
B^2=(A_{\cH}^2)\otimes A^2,
\qquad
(B_i')^2=\widehat Q_i\bigl((A_{\cH}^2)\otimes A^2\bigr)\widehat Q_i,
\]
we can compare the three parts explicitly.

First, for $x\in S$ and $y\in Y$,
\[
\lambda_{\cH\times G}((i,x),(i,y))
=\lambda_{\cH}^+(i,i)\lambda_G(x,y),
\]
whereas, by \eqref{eq:Bi-prime},
\[
\lambda_{(\cH\times G)'_i}((i,x),(i,y))
=\lambda_{\cH}^+(i,i)\lambda_{G'}(x,y).
\]
By (W2), it follows that $\mathcal M_1=\mathcal M_1'$.

Second, fix $h\ne i$.  For $x\in S$ and $y\in Y$,
\[
\lambda_{\cH\times G}((i,x),(h,y))
=\lambda_{\cH}^+(i,h)\lambda_G(x,y).
\]
On the other hand, by \eqref{eq:Bi-prime},
\[
\lambda_{(\cH\times G)'_i}((i,x),(h,y))
=\lambda_{\cH}^+(i,h)(A^{\mathrm r} A)_{xy}.
\]
Since $y\in Y$ and $Qe_y=e_y$, we have
\[
(A^{\mathrm r} A)_{xy}=(QA^2)_{xy}=(QA^2Q)_{xy}=((A')^2)_{xy}=\lambda_{G'}(x,y).
\]
Using (W2) again, the multiset corresponding to this fixed $h$ is unchanged.
Taking the disjoint union over all $h\ne i$ gives
$\mathcal M_2=\mathcal M_2'$.

It remains to compare $\mathcal M_3$ and $\mathcal M_3'$.  Put
$k_0:=d_G(x_0).$
Choose $\hat h\ne i$ such that
\[
\rho:=\lambda_{\cH}^+(i,\hat h)=\max_{h\ne i}\lambda_{\cH}^+(i,h)>0.
\]
Before switching,
\[
\lambda_{\cH\times G}((i,x_0),(\hat h,x_0))
=\rho\lambda_G(x_0,x_0)
=\rho k_0.
\]
Thus the value $\rho k_0$ occurs in $\mathcal M_3$.

After switching, for any $x,y\in S$ and any $h\ne i$,
\[
\lambda_{(\cH\times G)'_i}((i,x),(h,y))
=\lambda_{\cH}^+(i,h)(A^{\mathrm r} A)_{xy}.
\]
By \textup{(W3)}, and since $(A^{\mathrm r} A)_{xy}$ is a nonnegative integer, we have
$(A^{\mathrm r} A)_{xy}\le k_0-1$.  Moreover, $0\le\lambda_{\cH}^+(i,h)\le\rho$.  Hence
\[
\lambda_{(\cH\times G)'_i}((i,x),(h,y))
\le
\rho(k_0-1)
<
\rho k_0.
\]
So the value $\rho k_0$ does not occur in $\mathcal M_3'$.

Since $\mathcal M_1=\mathcal M_1'$ and $\mathcal M_2=\mathcal M_2'$, while
$\rho k_0$ occurs in $\mathcal M_3$ and not in $\mathcal M_3'$, we obtain
\[
\overline\Lambda_{\cH\times G}(X_i)
\ne
\overline\Lambda_{(\cH\times G)'_i}(X_i).
\]
This proves the asserted change of the external multiset.

Finally assume $|D_1|=|D_2|$.  Then, for the WQH partition of $\cH\times G$
displayed above,
\[
|D_1'|=|N_{\cH}(i)|\,|D_1|=|N_{\cH}(i)|\,|D_2|=|D_2'|.
\]
By Lemma~\ref{lem:external-test}~(c), the change of the external multiset implies
that $\cH\times G$ and $(\cH\times G)'_i$ are non-isomorphic.
\end{proof}

\section{WQH-switching for other graph matrices}\label{sec:other-matrices}

In this section we include some simple consequence of the matrix form of WQH-switching for several
standard graph matrices.  This is analogous to the corresponding discussion for
GM-switching in \cite[Section~3.3]{VANDAM2003241}.

Let $G$ be a WQH-type graph with adjacency matrix $A=A_G$, and let $G'$ be its
WQH-switched graph.  Let $Q$ be the WQH-switching matrix defined in
\eqref{eq:wqh-matrix}.  Thus
\[
A_{G'}=QA_GQ,
\qquad
Q^\top=Q,
\qquad
Q^2=I.
\]
Moreover, since each row of $Q$ has sum one, we have $Q\mathbf 1=\mathbf 1.$
Consequently,
\[
QJQ=J,
\qquad
QIQ=I.
\]

Let $D_G$ denote the diagonal degree matrix of $G$.  For real numbers
$\alpha,\beta,\gamma,\delta$, define
\[
M_{\alpha,\beta,\gamma,\delta}(G)
:=
\alpha A_G+\beta J+\gamma I+\delta D_G.
\]
When $\delta=0$, this is a generalized adjacency matrix in the usual sense.

\begin{prp}\label{prp:wqh-other-matrices}
	Let $G$ be a WQH-type graph, and let $G'$ be its WQH-switched graph.
	
	\begin{enumerate}[\textup{(\alph*)}]
		\item For all real numbers $\alpha,\beta,\gamma$,
		\[
		M_{\alpha,\beta,\gamma,0}(G')
		=
		Q M_{\alpha,\beta,\gamma,0}(G) Q.
		\]
		In particular, $G$ and $G'$ are cospectral with respect to every generalized
		adjacency matrix.
		
		\item Suppose, in addition, that $D_{G'}=QD_GQ.$
		Then, for all real numbers $\alpha,\beta,\gamma,\delta$,
		\[
		M_{\alpha,\beta,\gamma,\delta}(G')
		=
		Q M_{\alpha,\beta,\gamma,\delta}(G) Q.
		\]
		In particular, $G$ and $G'$ are cospectral with respect to the Laplacian
		matrix $L_G:=D_G-A_G$ and the signless Laplacian matrix $L_G^+:=D_G+A_G.$
		
		\item The condition in \textup{(b)} holds, for example, if all vertices in
		$C_1\cup C_2$ have the same degree in $G$.
	\end{enumerate}
\end{prp}

\begin{proof}
	For \textup{(a)}, using $A_{G'}=QA_GQ$, $QJQ=J$, and $QIQ=I$, we obtain
	\[
	\begin{aligned}
		Q M_{\alpha,\beta,\gamma,0}(G)Q
		&=
		Q(\alpha A_G+\beta J+\gamma I)Q
		=
		\alpha A_{G'}+\beta J+\gamma I=
		M_{\alpha,\beta,\gamma,0}(G').
	\end{aligned}
	\]
	Since $Q$ is orthogonal, the two matrices are similar and hence cospectral.
	
	For \textup{(b)}, the same computation gives
	\[
	\begin{aligned}
		Q M_{\alpha,\beta,\gamma,\delta}(G)Q
		=
		\alpha A_{G'}+\beta J+\gamma I+\delta D_{G'}=
		M_{\alpha,\beta,\gamma,\delta}(G').
	\end{aligned}
	\]
	Taking $(\alpha,\beta,\gamma,\delta)=(-1,0,0,1)$ gives $L_G=D_G-A_G$, while
	taking $(\alpha,\beta,\gamma,\delta)=(1,0,0,1)$ gives $L_G^+=D_G+A_G$.
	
	It remains to prove \textup{(c)}.  Suppose that all vertices in
	$C_1\cup C_2$ have the same degree in $G$.  Since $Q$ acts nontrivially
	only on the coordinates indexed by $C_1\cup C_2$, the degree matrix $D_G$
	commutes with $Q$, and hence $QD_GQ=D_G.$  Moreover,
	\[
	A_{G'}\mathbf 1
	=
	QA_GQ\mathbf 1
	=
	QA_G\mathbf 1
	=
	QD_G\mathbf 1
	=
	D_G\mathbf 1.
	\]
	Thus $G'$ has the same degree matrix as $G$, namely
	$D_{G'}=D_G=QD_GQ.$
	This proves \textup{(c)}.
\end{proof}

\section{Applications}\label{sec:applications}

We conclude with several consequences of the preceding non-isomorphism criteria from Theorem~\ref{thm:clique-wqh} and \ref{thm:WQH-weak-tensor}.

\begin{itemize}
    \item Let $G$ be a WQH-type graph satisfying \textup{(W1)--(W3)} of
Theorem~\ref{thm:WQH-weak-tensor}, and suppose that $|D_1|=|D_2|$.  If
$\mathcal H$ has two distinct vertices $i,j$ with $\lambda_{\mathcal H}^{+}(i,j)>0,$
then Theorem~\ref{thm:WQH-weak-tensor} gives a cospectral mate
$(\mathcal H\times G)'_i$ of $\mathcal H\times G$, and the two graphs are
non-isomorphic.  Hence this product graph is not determined by its adjacency
spectrum.  In particular, one may take $\mathcal H=P_3$, with $i,j$ the two
end vertices, or $\mathcal H=K_m$ with $m\ge3$.
    \item Theorem~\ref{thm:WQH-weak-tensor} also gives applications to coclique extensions.  Let
$\mathcal K_t^\circ$ denote the complete graph on $t$ vertices with a loop at
every vertex.  Then $\mathcal K_t^\circ\times G$ has adjacency matrix
$J_t\otimes A_G$, and hence is naturally isomorphic to the $t$-coclique
extension of $G$.  Since
\[
\lambda_{\mathcal K_t^\circ}^{+}(i,j)=t>0
\qquad (i\ne j),
\]
Theorem~\ref{thm:WQH-weak-tensor} shows that, for every $t\ge2$, the
$t$-coclique extension of such a graph $G$ has a non-isomorphic cospectral
mate.

    \item Theorem~\ref{thm:clique-wqh} gives another family of examples.  Since
Theorem~\ref{thm:clique-wqh} holds for $t\ge2$, it applies in particular to
$2$-clique extensions.  The clique-extension constructions in
\cite{GDK2026+} fit into this setting: the untwisted graph is a $2$-clique
extension of the corresponding base graph, and the elementary twists are
realized by WQH-switchings.  The resulting switched graphs are cospectral with
the untwisted graph, while the non-isomorphism is detected by the
common-neighbour data appearing in Theorem~\ref{thm:clique-wqh}.

\item Finally, one known WQH-switching construction has a non-isomorphism proof of
the same common-neighbour type.  In \cite[Theorem~13]{ABIAD20231}, the
generalized Johnson graph $J_{\{2\}}(n,4)$ is edge-regular, whereas after
WQH-switching one obtains an adjacent pair with more common neighbours than any
adjacent pair in the original graph.  Thus the non-isomorphism proof there is
based on the same type of common-neighbour obstruction as
Lemma~\ref{lem:external-test}, restricted to adjacent pairs.
 
\end{itemize}

\subsection*{Acknowledgments}

Aida Abiad is supported by NWO (Dutch Research Council) through the grant
VI.Vidi.213.085.  Hong-Jun Ge is supported by the CSC Scholarship Program
(No.~202506340038).

\bibliographystyle{abbrv}
\bibliography{AG}

@article{AHswitching,
  title={Cospectral graphs and regular orthogonal matrices of level 2},
  volume={19},
  issue={3},
  author={A. Abiad and W. H. Haemers},
  journal={Electron. J. Comb.},
  pages={\#P13},
  year={2012}
}

@article{GDK2026+,
  author  = {van Dam, E. R. and Ge, H.-J. and Koolen, J. H.},
  title   = {Co-edge-regular four-eigenvalue graphs with unbounded coherent rank},
  journal = {work in progress},
  year    = {2026}
  }

@article {abh2015,
    AUTHOR = {Abiad, Aida and Brouwer, Andries E. and Haemers, Willem H.},
     TITLE = {Godsil-{M}c{K}ay switching and isomorphism},
   JOURNAL = {Electron. J. Linear Algebra},
  FJOURNAL = {The Electronic Journal of Linear Algebra},
    VOLUME = {28},
      YEAR = {2015},
     PAGES = {4--11},
      ISSN = {1081-3810},
   MRCLASS = {05C50},
  MRNUMBER = {3386383},
       DOI = {10.13001/1081-3810.2986},
       URL = {https://doi.org/10.13001/1081-3810.2986},
}

@article {MRC,
AUTHOR = {Abiad, Aida and Brimkov, Boris and Breen, Jane and Cameron,     Thomas R. and Gupta, Himanshu and Villagr\'an, Ralihe R.},
     TITLE = {Constructions of cospectral graphs with different zero forcing
              numbers},
   JOURNAL = {Electron. J. Linear Algebra},
  FJOURNAL = {The Electronic Journal of Linear Algebra},
    VOLUME = {38},
      YEAR = {2022},
     PAGES = {280--294},
      ISSN = {1081-3810},
   MRCLASS = {05C50 (15A18)},
  MRNUMBER = {4426610},
MRREVIEWER = {Jephian\ C.-H.\ Lin},
}

@article{h2020,
    author = {Haemers, W. H.},
    title = {Cospectral pairs of regular graphs with different connectivity},
    journal = {Discuss. Math. Graph Theory},
    volume = {40},
    year = {2020},
    pages = {577--584}
}

@article{bch2015,
title = {Cospectral regular graphs with and without a perfect matching},
journal = {Discrete Math.},
volume = {338},
number = {3},
pages = {199-201},
year = {2015},
issn = {0012-365X},
doi = {https://doi.org/10.1016/j.disc.2014.11.002},
url = {https://www.sciencedirect.com/science/article/pii/S0012365X14004099},
author = {Zoltán L. Blázsik and Jay Cummings and Willem H. Haemers}
}

@article{EVANS2023113384,
title = {A general construction of strictly Neumaier graphs and a related switching},
journal = {Discrete Math.},
volume = {346},
number = {7},
pages = {113384},
year = {2023},
issn = {0012-365X},
doi = {https://doi.org/10.1016/j.disc.2023.113384},
url = {https://www.sciencedirect.com/science/article/pii/S0012365X23000705},
author = {Rhys J. Evans and Sergey Goryainov and Elena V. Konstantinova and Alexander D. Mednykh}
}

@article{IHRINGER2021112560,
title = {Graphs cospectral with ${N}{U}(n+1,q^2)$, $n\neq 3$},
journal = {Discrete Math.},
volume = {344},
number = {11},
pages = {112560},
year = {2021},
issn = {0012-365X},
doi = {https://doi.org/10.1016/j.disc.2021.112560},
url = {https://www.sciencedirect.com/science/article/pii/S0012365X21002739},
author = {Ferdinand Ihringer and Francesco Pavese and Valentino Smaldore}
}

@article{VANDAM2003241,
title = {Which graphs are determined by their spectrum?},
journal = {Linear Algebra Appl.},
volume = {373},
pages = {241-272},
year = {2003},
note = {Combinatorial Matrix Theory Conference (Pohang, 2002)},
issn = {0024-3795},
doi = {https://doi.org/10.1016/S0024-3795(03)00483-X},
url = {https://www.sciencedirect.com/science/article/pii/S002437950300483X},
author = {Edwin R. {van Dam} and Willem H. Haemers}
}

@article{ABScounting,
title = {Counting cospectral graphs obtained via switching},
journal = {Discrete Math.},
volume = {349},
number = {3},
pages = {114775},
year = {2026},
issn = {0012-365X},
doi = {https://doi.org/10.1016/j.disc.2025.114775},
url = {https://www.sciencedirect.com/science/article/pii/S0012365X25003838},
author = {Aida Abiad and Nils {van de Berg} and Robin Simoens}
}

@article{MAO2023,
title = {Constructing cospectral graphs via regular rational orthogonal matrices with level two},
journal = {Discrete Math.},
volume = {346},
number = {1},
pages = {113156},
year = {2023},
author = {L. Mao and W. Wang and F. Liu and L. Qiu}
}

@article{ABIAD20231,
title = {Cospectral mates for generalized {J}ohnson and {G}rassmann graphs},
journal = {Linear Algebra Appl.},
volume = {678},
pages = {1-15},
year = {2023},
issn = {0024-3795},
doi = {https://doi.org/10.1016/j.laa.2023.08.015},
url = {https://www.sciencedirect.com/science/article/pii/S0024379523003208},
author = {Aida Abiad and Jozefien D'haeseleer and Willem H. Haemers and Robin Simoens}
}

@article{h1996,
title = {Distance-regularity and the spectrum of graphs},
journal = {Linear Algebra Appl.},
volume = {236},
pages = {265-278},
year = {1996},
issn = {0024-3795},
doi = {https://doi.org/10.1016/0024-3795(94)00166-9},
url = {https://www.sciencedirect.com/science/article/pii/0024379594001669},
author = {Willem H. Haemers}
}

@article{abiad2024,
      title={Switching methods of level 2 for the construction of cospectral graphs}, 
      author={Aida Abiad and Nils van de Berg and Robin Simoens},
      year={2026},
      journal={Linear Algebra Appl.},
      volume={to appear}
}

@article {GM82,
    AUTHOR = {Godsil, C. D. and McKay, B. D.},
     TITLE = {Constructing cospectral graphs},
   JOURNAL = {Aequationes Math.},
  FJOURNAL = {Aequationes Mathematicae},
    VOLUME = {25},
      YEAR = {1982},
    NUMBER = {2-3},
     PAGES = {257--268},
      ISSN = {0001-9054,1420-8903},
   MRCLASS = {05C50},
  MRNUMBER = {730486},
MRREVIEWER = {M.\ Doob},
       DOI = {10.1007/BF02189621},
       URL = {https://doi.org/10.1007/BF02189621},
}

@article {IM19,
    AUTHOR = {Ihringer, Ferdinand and Munemasa, Akihiro},
     TITLE = {New strongly regular graphs from finite geometries via
              switching},
   JOURNAL = {Linear Algebra Appl.},
  FJOURNAL = {Linear Algebra and its Applications},
    VOLUME = {580},
      YEAR = {2019},
     PAGES = {464--474},
      ISSN = {0024-3795,1873-1856},
   MRCLASS = {51E20 (05C50 05E30)},
  MRNUMBER = {3982621},
MRREVIEWER = {Atif\ A.\ Abueida},
       DOI = {10.1016/j.laa.2019.07.014},
       URL = {https://doi.org/10.1016/j.laa.2019.07.014},
}

@article{ihringer2025,
      title={Design switching on graphs}, 
      author={Ferdinand Ihringer and Robin Simoens},
      year={2025},
      volume={2508.11523},
      journal={arXiv},
      primaryClass={math.CO},
      url={https://arxiv.org/abs/2508.11523}, 
}

@article{QIU2020265,
title = {On a theorem of Godsil and McKay concerning the construction of cospectral graphs},
journal = {Linear Algebra Appl.},
volume = {603},
pages = {265-274},
year = {2020},
author = {L. Qiu and Y. Ji and W. Wang}
}

@article {QJW20,
    AUTHOR = {Qiu, Lihong and Ji, Yizhe and Wang, Wei},
     TITLE = {On a theorem of {G}odsil and {M}c{K}ay concerning the
              construction of cospectral graphs},
   JOURNAL = {Linear Algebra Appl.},
  FJOURNAL = {Linear Algebra and its Applications},
    VOLUME = {603},
      YEAR = {2020},
     PAGES = {265--274},
      ISSN = {0024-3795,1873-1856},
   MRCLASS = {05C50},
  MRNUMBER = {4110279},
MRREVIEWER = {Jinsong\ Chen},
       DOI = {10.1016/j.laa.2020.05.025},
       URL = {https://doi.org/10.1016/j.laa.2020.05.025},
}

@article {WQH2019,
    AUTHOR = {Wang, Wei and Qiu, Lihong and Hu, Yulin},
     TITLE = {Cospectral graphs, {GM}-switching and regular rational
              orthogonal matrices of level {$p$}},
   JOURNAL = {Linear Algebra Appl.},
  FJOURNAL = {Linear Algebra and its Applications},
    VOLUME = {563},
      YEAR = {2019},
     PAGES = {154--177},
      ISSN = {0024-3795,1873-1856},
   MRCLASS = {05C50},
  MRNUMBER = {3872985},
MRREVIEWER = {Milica\ Andelic},
       DOI = {10.1016/j.laa.2018.10.027},
       URL = {https://doi.org/10.1016/j.laa.2018.10.027},
}

@article {ih2022,
    AUTHOR = {Ihringer, Ferdinand},
     TITLE = {Switching for small strongly regular graphs},
   JOURNAL = {Australas. J. Combin.},
  FJOURNAL = {The Australasian Journal of Combinatorics},
    VOLUME = {84},
      YEAR = {2022},
     PAGES = {28--48},
      ISSN = {1034-4942,2202-3518},
   MRCLASS = {05E30},
  MRNUMBER = {4468875},
MRREVIEWER = {Marija\ Maksimovi\'c},
}

\end{document}